\documentclass{amsart}
 \newtheorem{thm}{Theorem}[section]
 
 \newtheorem{lem}[thm]{Lemma}
 \newtheorem{prop}[thm]{Proposition}
 \theoremstyle{definition}
 \newtheorem{defn}[thm]{Definition}
 \theoremstyle{remark}
 \newtheorem{rem}[thm]{Remark}
 \newtheorem*{ex}{Example}
 \numberwithin{equation}{section}

\begin{document}

\title[Cyclic vectors in the ball]
 {A note on Dirichlet-type spaces and cyclic vectors in the unit ball of 
$\mathbb{C}^2$}
\author[Sola]{Alan Sola}

\address{
Department of Pure Mathematics and Mathematical Statistics\\University of Cambridge\\
Wilberforce Road, Cambridge CB3 0WB\\UK}

\email{a.sola@statslab.cam.ac.uk}

%\thanks{This work was completed with the support of our



\begin{abstract}
We extend results of B\'en\'eteau, Condori, Liaw, Seco, and Sola 
concerning cyclic vectors in Dirichlet-type spaces to the setting of the 
unit ball, identifying some classes of cyclic and non-cyclic
functions, and noting the 
necessity of certain capacity conditions. 
\end{abstract}

%%% ----------------------------------------------------------------------
\maketitle
%%% ----------------------------------------------------------------------
%\tableofcontents

%\begin{document}
\bibliographystyle{alpha}

%\keywords{Dirichlet-type spaces, cyclic vectors, capacities for 
%the sphere.}
%\subjclass[2010]{Primary: 32A36. Secondary: 31C25, 47A16.}
%\begin{abstract}
%We extend results of B\'en\'eteau, Condori, Liaw, Seco, and Sola 
%concerning cyclic vectors in Dirichlet-type spaces to the setting of the 
%unit ball. We identify some classes of cyclic and non-cyclic
%functions, and establish the 
%necessity of certain capacity conditions. 
%\end{abstract}

%\maketitle

% -------------------------------------------------------------------------
\section{Introduction and preliminaries}
\subsection{Dirichlet-type spaces and cyclic vectors}
Let \[\mathbb{B}^2=\{(z_1,z_2)\in \mathbb{C}^2\colon |z_1|^2+|z_2|^2<1\}\]
denote the {\it unit ball} in $\mathbb{C}^2$, and let
\[\mathbb{S}^2=\partial \mathbb{B}^2=\{(\zeta_1,\zeta_2)\in 
\mathbb{C}^2\colon |\zeta_1|^2+|\zeta_2|^2=1\}\]
be its boundary, the {\it unit sphere}. In this note, we are concerned with certain Hilbert spaces of analytic 
functions defined on the ball. Let $\alpha\in (-\infty, \infty)$ be fixed. 
An analytic 
function $f\colon \mathbb{B}^2\to 
\mathbb{C}$ having power series expansion
\begin{equation}
f(z_1,z_2)=\sum_{k=0}^{\infty}\sum_{l=0}^{\infty}a_{k,l}z_1^kz_2^l
\label{powerexp}
\end{equation}
is said to belong to the {\it Dirichlet-type space} $\mathcal{D}_{\alpha}$ if 
\begin{equation}
\|f\|^2_{\alpha}=\sum_{k=0}^{\infty}\sum_{l=0}^{\infty}
(2+k+l)^{\alpha}\frac{k!l!}{(1+k+l)!}|a_{k,l}|^2<\infty.
\label{alphanormdef}
\end{equation}
General introductions to function theory in the ball can be found in \cite{RudBook, ZhuBook}.

%A few remarks are in order. 
\begin{rem}Since 
the ball is a connected Reinhardt domain containing the origin,
%(see \cite[Chapter 2]{HorBook})
any analytic function in $\mathbb{B}^2$ has a 
power series expansion of the form \eqref{powerexp} that is valid in 
the unit ball. 
\end{rem}
%\begin{rem}
The spaces we consider here represent one possible 
generalization to two variables of the one-variable {\it Dirichlet-type 
spaces} $D_{\alpha}$ (see \cite{EKMRBook}) consisting of analytic functions 
$f(z)=\sum_{k=0}^{\infty}a_kz^k$ in the unit disk having 
\[\|f\|^2_{D_{\alpha}}=\sum_{k=0}^{\infty}(k+1)^{\alpha}|a_k|^2<\infty.\]
The weights in the $\ell^2$ norm in \eqref{alphanormdef} are chosen in 
such a way that $\mathcal{D}_{0}$ and $\mathcal{D}_{-1}$ coincide 
with the usual {\it Hardy} and {\it Bergman spaces} of the ball, as 
in the one-variable setting. The norm in 
these spaces is usually defined via integrals over spheres and balls, but 
using \cite[Prop. 1.4.9]{RudBook} one readily verifies that 
\eqref{alphanormdef} furnishes an equivalent norm for these spaces.
The {\it Dirichlet space}, distinguished by its 
M\"obius invariance properties and discussed in \cite{ZhuBook}, corresponds 
to the parameter choice $\alpha=2$ (in the unit disk and unit bidisk, we get 
the Dirichlet space when $\alpha=1$). Some background material, including characterizations of $\mathcal{D}_{\alpha}$ involving radial 
derivatives, can be found in
\cite{AC89, ZhuBook,Li06, BB08, Mi11} and the references provided 
therein. 
%\end{rem}
When $\alpha>2$, the spaces $\mathcal{D}_{\alpha}$ are algebras; we will 
focus on the case $\alpha\leq 2$. It is apparent from the definition of the norm \eqref{alphanormdef} that 
$\mathcal{D}_{\alpha}\subset \mathcal{D}_{\beta}$ if $\alpha\geq \beta$, 
and that polynomials in two complex variables 
are dense in all $\mathcal{D}_{\alpha}$.
%\subsection{Cyclic vectors}

We say that $f\in \mathcal{D}_{\alpha}$ is a {\it cyclic vector} 
if the closed invariant subspace
\[[f]=\mathrm{clos}\, \, \mathrm{span}\{z_1^kz_2^lf\colon k=0,1,\ldots, l=0,1,\ldots\}\]
coincides with $\mathcal{D}_{\alpha}$. As usual, invariance refers to joint 
invariance under the bounded linear operators $\{S_1,S_2\}$ induced by multiplication by 
the coordinate functions: $S_j\colon f\mapsto z_j\cdot f$, $j=1,2$.

The basic example of a cyclic function is $f(z_1,z_2)=1$; its cyclicity is a consequence of the density of polynomials in $\mathcal{D}_{\alpha}$. An immediate consequence is that $f\in \mathcal{D}_{\alpha}$ is cyclic if and only if there exists a sequence $(p_n)_{n}$ of polynomials in two variables such that 
$\|p_nf-1\|_{\alpha}\to 0$ as $n \to \infty$.
At the other extreme, if 
a function $f\in \mathcal{D}_{\alpha}$ vanishes at some point $(z_1,z_2)\in \mathbb{B}^2$, then $f$ is 
not cyclic because $\mathcal{D}_{\alpha}$ enjoys the {\it bounded 
point evaluation property} (BPE), and hence 
elements of the invariant subspace generated by a given function inherit all 
its zeros in $\mathbb{B}^2$. If $\alpha>2$, then $f\in \mathcal{D}_{\alpha}$ is cyclic precisely 
when $f$ does not vanish on $\overline{\mathbb{B}^2}$.

The purpose of this note is to continue the investigations carried out in the 
setting of the bidisk in \cite{BCLSSII13, BKKLSS} 
by studying the spaces $\mathcal{D}_{\alpha}$. The ball $\mathbb{B}^2$
and the bidisk $\mathbb{D}^2=\{(z_2,z_2)\in \mathbb{C}^2\colon |z_1|<1, 
|z_2|<1\}$ are 
the most obvious two-dimensional analogs of the unit disk, 
but they are different in a number of ways as regards geometry and function theory.  For instance, $\partial \mathbb{D}^2$ contains an abundance of analytic disks, whereas the unit sphere $\mathbb{S}^2$ contains no (non-trivial) analytic disks at all. Another important difference between the ball and the bidisk is that the topological and Shilov boundaries of the ball coincide, while the Shilov boundary of the bidisk can be identified with the torus $\mathbb{T}^2$, which is much smaller than $\partial \mathbb{D}^2$.
Nevertheless, the general flavor of the results and observations in this note is the same as in the bidisk
paper \cite{BCLSSII13}, and the methods we use are essentially the same. For 
instance,
\begin{itemize}
\item we are able to observe quantitative differences in how fast 
$1/f$ can be approximated by polynomials for different classes of cyclic 
functions,
\item there exist polynomials in two complex variables that do not vanish in 
$\mathbb{B}^2$, but are not cyclic in $\mathcal{D}_{\alpha}$ for certain 
values of $\alpha$, and
\item for an appropriate notion of capacity,
$\mathrm{cap}_{\alpha}(\mathcal{Z}(f)\cap \mathbb{S}^2)>0$ implies that 
$f\in \mathcal{D}_{\alpha}\cap A(\mathbb{B}^2)$ is {\it not} cyclic.
\end{itemize}
\subsection{Preliminaries}
Throughout, we shall write 
$\langle z, w\rangle=z_1\bar{w}_1+z_2\bar{w}_2$ for the usual Euclidean inner 
product on $\mathbb{C}^2$. We will use $d\sigma(\zeta_1,\zeta_2)$ to indicate 
integration with respect to the (induced) normalized surface measure on the 
sphere.

The following is standard, see \cite[Prop. 1.4.9]{RudBook} or 
\cite[Lemma 1.11]{ZhuBook}.
\begin{lem}
The monomials $\{z_1^kz_2^l\}_{k,l\in \mathbb{N}}$ form an orthogonal set with respect to area measure in the ball and surface measure on the sphere, and
\[\int_{\mathbb{S}^2}|\zeta_1^k\zeta_2^l|^2d\sigma(\zeta_1,\zeta_2)=\frac{k!l!}{(1+k+l)!}.\]
\end{lem}
The latter formula appears naturally as part of the $\ell^2$ weight in the 
norm definition \eqref{alphanormdef}. (In particular, monomials do not 
have norm $1$ in $H^2(\mathbb{B}^2)$.)

One way of building analytic functions in $\mathbb{B}^2$ is to take the {\it Cauchy transform}
of functions or measures on the sphere,
\begin{equation}
C[\mu](z_1,z_2)=\int_{\mathbb{S}^2}\frac{d\mu(\zeta_1,\zeta_2)}{(1-\langle z,\overline{\zeta}\rangle)^2}, 
\quad (z_1,z_2) \in \mathbb{B}^2.
\end{equation}
(Our definition of the Cauchy transform differs from the usual one in featuring
a complex conjugate; this will turn out to be convenient later on.)

If $f$ belongs to $A(\mathbb{B}^2)$, the {\it ball algebra}, and if we let 
$d\mu/d\sigma=f|_{\mathbb{S}^2}$, then the radial limits of $C[\mu]$ 
coincide with the values of $f$ on the sphere. The natural stronger notion of convergence in the ball, ie. the analog of 
non-tangential convergence, is given by the so-called 
{\it Kor\'anyi-} or {\it $K$-limit}, taken over approach regions of the form
\[\Gamma_{a}(\zeta)=\left\{(z_1,z_2)\in \mathbb{B}^2\colon |1-\langle z,\zeta\rangle|<\frac{a}{2}(1-|z|^2)\right\} \quad (\zeta \in \mathbb{S}^2, a>1).\]
See \cite[Chapter 5]{RudBook} for definitions and background material. 
In particular, it is known that any $f\in H^2(\mathbb{B}^2)$ has finite
$K$-limits, and hence finite radial limits, almost everywhere on 
$\mathbb{S}^2$. The elements of $\mathcal{D}_{\alpha}$, $\alpha>0$, 
then inherit this property. We shall return to exceptional sets 
towards the end of this note.

The spaces $\mathcal{D}_{\alpha}$ are Hilbert spaces, hence self-dual, but they 
also admit a version of the $(D_{\alpha})^*\cong D_{-\alpha}$ duality considered by 
Brown and Shields in \cite{BS84}. The dual pairing in question is 
given by 
\begin{equation}
\langle f,g \rangle=\sum_{k=0}^{\infty}\sum_{l=0}^{\infty}
\frac{k!l!}{(1+k+l)!}a_{k,l}b_{k,l}=\lim_{r\to 1}\int_{r\mathbb{S}^2}
f(\zeta_1,\zeta_2)g(\bar{\zeta_1},\bar{\zeta_2})d\sigma(\zeta_1,\zeta_2),
\label{dualpairing}
\end{equation}
for $f=\sum_{k,l}a_{k,l}z_1^kz_2^l\in \mathcal{D}_{\alpha}$ and 
$g=\sum_{k,l}b_{k,l}z_1^kz_2^l\in \mathcal{D}_{-\alpha}$. The 
Cauchy-Schwarz inequality guarantees that this is indeed well-defined. 

In what follows, we shall also need to consider spaces of analytic functions in one complex variable. In addition to the usual Dirichlet-type spaces $D_{\alpha}$ of the unit disk, we need the space $d_{\alpha}$
consisting of analytic functions $f\colon (1/\sqrt{2})\mathbb{D}=
\{z\in \mathbb{C}\colon |z|<1/\sqrt{2}\}\to \mathbb{C}$ having 
\[\|f\|_{d_{\alpha}}^2=\int_{0}^{1/\sqrt{2}}\int_{0}^{2\pi}|f'(re^{i\theta})|^2|(1-2r^2)^{1-\alpha}
rdrd\theta<\infty.\]
An equivalent norm for $d_{\alpha}$ can be given in terms of Taylor coefficients:
\[\|f\|^2_{d_{\alpha}}\asymp \sum_{k=0}^{\infty}2^{-k}(k+1)^{\alpha}|a_k|^2.\]
There is a natural identification between function theories of $D_{\alpha}$ and $d_{\alpha}$, and one verifies that, {\it mutatis mutandis}, 
the results in \cite{BCLSS13} are valid for $d_{\alpha}$.

Let $P_n$ denote the vector spaces of polynomials in $z$ having degree at most $n$, and let $\varphi_{\alpha}(n)=n^{1-\alpha}$ (where $n=1,2,\ldots$) for $\alpha \in [0,1)$, and take $\varphi_1(n)=\sum_{j=1}^nj^{-1}$. In \cite{BCLSS13}, the sharp estimate 
\begin{equation}
\mathrm{dist}_{D_{\alpha}}(1,f\cdot P_n)=\inf_{p\in P_n}\|p\cdot f-1\|_{D_{\alpha}}
\asymp \varphi_{\alpha}^{-1}(n+1) \quad (n\to \infty)
\label{sharpnorm}
\end{equation}
was obtained for $f\in D_{\alpha}$ that have
no zeros inside the unit disk, admit an analytic continuation to 
a strictly bigger disk, and have at least one zero on $\mathbb{T}$ (see \cite[Thm 3.7]{BCLSS13}). In what follows, we shall refer to this as condition NZAC.
Functions in $D_{\alpha}$ satisfying NZAC were previously known to be cyclic and 
the novelty in \cite{BCLSS13} was the determination of the optimal rate of 
decay for $\mathrm{dist}_{D_{\alpha}}(1, f\cdot P_n)$.
The reason for requiring $\mathcal{Z}(f)\cap \mathbb{T}\neq \emptyset$ is 
that any function that is non-vanishing on the {\it closed} unit disk is 
cyclic, and the associated quantity 
$\mathrm{dist}_{D_{\alpha}}(1,f\cdot P_n)$ exhibits exponential 
decay, or better; this can be seen by examining the truncated power series of $1/f$.
%\end{rem}
\section{One-variable functions and small zero sets}
One of the main differences between the ball and the bidisk is that the latter 
is a product domain, while the former is not, and the spaces $\mathcal{D}_{\alpha}$ are not tensor products of 
one-variable Dirichlet spaces: if $f=f(z_1)$ and $g=g(z_2)$, then, 
in general, $\|E_1(f)E_2(g)\|_{\alpha} \neq \|f\|_{D_{\alpha}}\cdot \|g\|_{D_{\alpha}}$. (Example: $f=z_1$, $g=z_2$.) Here, $E_j$ ($j=1,2$) are operators of extension:
given a one-variable function, set
\[E_j(f)(z_1,z_2)=f(z_j),\quad (z_1,z_2)\in \mathbb{B}^2.\]
We also note that the unit disk can be identified with the subset
\[\{(z_1,z_2)\in \mathbb{C}^2\colon |z_1|<1, z_2=0\}\subset \mathbb{B}^2,\]
and hence the operator of restriction, 
\[R_1(f)(z_1)=f(z_1,0), \quad z_1\in \mathbb{D},\]
is well-defined, and $R_1\circ E_1$ is the identity operator.

In what follows, $\mathcal{P}_n$ denotes the space of two-variable polynomials of degree at most $2n$. 
\begin{thm}\label{onevarthm}
Let $\alpha\leq 2$, and suppose $f\in D_{\alpha-1}$ satisfies condition NZAC. Then the function $E_1(f)$ is cyclic in $\mathcal{D}_{\alpha}$, and 
$\mathrm{dist}^2_{\mathcal{D}_{\alpha}}(1,f\cdot \mathcal{P}_n)\asymp \varphi^{-1}_{\alpha-1}(n+1)$.
\end{thm}
If $f, g\in D_{\alpha-1}$ satisfy NZAC and 
both $E_1(f)$ and $E_2(g)$ are multipliers, then the product of 
$E_1(f)(z_1,z_2)=f(z_1)$ and $E_2(g)(z_1,z_2)=g(z_2)$
 is cyclic for all $\alpha\leq 2$. 
We note that a function that satisfies the hypotheses above is not cyclic in any $\mathcal{D}_{\alpha}$ with $\alpha>2$ since it vanishes in $\overline{\mathbb{B}^2}$. Theorem \ref{onevarthm} is a straight-forward consequence of \eqref{sharpnorm} and the following easy norm comparison (no doubt well-known to experts, see \cite{RudBook} for the Hardy/Bergman case).
\begin{prop}
If $f \in D_{\alpha-1}$, then $\|E_1(f)\|_{\alpha}\asymp \|f\|_{D_{\alpha-1}}$.
\end{prop}
\begin{proof}
This is completely elementary: we have
\[\|z_1\|^2_{\alpha}=(2+k)^{\alpha}\frac{k!}{(1+k)!}\asymp (1+k)^{\alpha}\cdot\frac{1}{1+k}=(k+1)^{\alpha-1},\]
and the result follows by orthogonality of monomials.
\end{proof}
\begin{ex}
Consider the functions $f(z_1,z_2)=1-z_1$ and $g(z_1,z_2)=(1-z_1)(1-z_2)$. 
None of these functions vanishes in the ball, and we have
\[\mathcal{Z}(f)\cap \mathbb{S}^2=\{(1,0)\} \quad 
\textrm{and}\quad \mathcal{Z}(g)\cap \mathbb{S}^2=\{(1,0),(0,1)\},\]
so that both zero sets are discrete 
(unlike in the bidisk where $\mathcal{Z}(f)$ gives rise to an analytic disk in 
$\partial \mathbb{D}^2$). 
By Theorem \ref{onevarthm}, 
the functions $f$ and $g$ are cyclic in $\mathcal{D}_{\alpha}$ for 
$\alpha\leq 2$, 
but are not cyclic for $\alpha>2$.

Next, we examine the polynomial 
$f^{\sharp}(z_1,z_2)=1-(z_1+z_2)/\sqrt{2}$, which is irreducible and has
$\mathcal{Z}(f^{\sharp})\cap \mathbb{S}^2=\{(2^{-1/2},2^{-1/2})\}.$
This function is not a product of one-variable functions, but letting
\begin{equation}
U=\left(\begin{array}{cc}\frac{1}{\sqrt{2}} & \frac{1}{\sqrt{2}}\\ 
\frac{1}{\sqrt{2}} &-\frac{1}{\sqrt{2}}\end{array}\right),
\label{unitary}
\end{equation}
we note that $f^{\sharp}(z_1,z_2)=f(U(z_1,z_2)^T)=(f\circ U)(z_1,z_2)$.
Now if $p$ is a polynomial, then so is $p\circ U^{-1}$, and
\[\|p\cdot f^{\sharp}-1\|_{\alpha}=\|(p\circ U^{-1})\cdot f-1\|_{\alpha},\]
since the $\mathcal{D}_{\alpha}$ norm is invariant under unitary 
transformations. Thus the quantity $\mathrm{dist}_{\alpha}(1,f^{\sharp}\cdot \mathcal{P}_n)$
is of the same order as for $f$, and in particular $f^{\sharp}$ 
is cyclic in $\mathcal{D}_{\alpha}$ for $\alpha\leq 2$.
\end{ex}

The preceding example shows that $\mathrm{dist}^2_{\alpha}(1,f\cdot \mathcal{P}_n)\asymp \varphi^{-1}_{\alpha-1}(n+1)$ 
can occur also for functions of two variables whose zero sets are small, so 
that this particular rate of decay 
reflects the size of $\mathcal{Z}(f)$ rather than 
the algebraic property of being a polynomial in one variable only. 

We refer the reader to \cite{Hed89} for more results on functions in the ball algebra that vanish at a single point of $\mathbb{S}^2$. 
\section{Diagonal subspaces and zero curves}
We now identify further classes of cyclic functions.
We let $\mathcal{J}_{\alpha}$ denote the closed subspace of $\mathcal{D}_{\alpha}$ consisting of functions 
of the form
\[f(z_1,z_2)=\sum_{k=0}^{\infty}a_k(z_1z_2)^k.\]
Given $f\in \mathcal{J}_{\alpha}$, we can 
produce a one-variable function by applying the operator
$\mathcal{R}\colon \mathrm{Hol}(\mathbb{B}^2)\to \mathrm{Hol}((1/\sqrt{2})\mathbb{D})$ 
which takes $f$ to
\[\mathcal{R}(f)(z_1)=f(z_1/\sqrt{2}), \quad z_1\in (1/\sqrt{2})\mathbb{D}.\]
This is well-defined because
\[2|z_1z_2|=|z_1|^2+|z_2|^2-(|z_1|-|z_2|)^2\leq 1\]
for all $(z_1,z_2)\in \overline{\mathbb{B}^2}$.
Geometrically speaking, we are looking at a disk embedded in the ball--but not in a coordinate plane--and then restricting $f$ to that disk. Similarly, a function $f\in d_{\alpha}$ can be mapped into $\mathcal{D}_{\alpha+1/2}$ via
\[\mathcal{L}(f)(z_1,z_2)=f(\sqrt{2}z_1\cdot z_2), \quad (z_1,z_2)\in \mathbb{B}^2.\]
Composing $\mathcal{R}$ and $\mathcal{L}$ yields the respective identity 
operators on $d_{\alpha}$ and $\mathcal{J}_{\alpha}$.
\begin{thm}\label{diagspaces}
Suppose $f\in \mathcal{J}_{\alpha}$, and suppose $\mathcal{R}(f)$ satisfies 
condition NZAC. Then $f$
is cyclic in $\mathcal{D}_{\alpha}$ precisely when $\alpha\leq 3/2$, and 
$\mathrm{dist}^2_{\mathcal{D}_{\alpha}}(1,f\cdot \mathcal{P}_{n})\asymp \varphi^{-1}_{\alpha-1/2}(n+1)$.
\end{thm}
Theorem \ref{diagspaces} follows from the results in \cite{BCLSS13}, adapted to $d_{\alpha-1/2}$,
and norm comparisons between two-variable and one-variable spaces; 
the high-level arguments are described in \cite{BCLSSII13}. The 
crucial point is the following result, which is a counterpart to 
\cite[Lemma 3.3]{BCLSSII13}.
\begin{lem}
If $f\in \mathcal{J}_{\alpha}$, then $\|f\|_{\alpha}\asymp \|\mathcal{R}(f)\|_{d_{\alpha-1/2}}$.
\end{lem}
\begin{proof}
If $f\in \mathcal{J}_{\alpha}$, then
\begin{equation*}
\|f\|^2_{\alpha}=\sum_{k=0}^{\infty}(2+2k)^{\alpha}\frac{(k!)^2}{(1+2k)!}|a_{k}|^2
\asymp \sum_{k=0}^{\infty}(k+1)^{\alpha}\left[\frac{(k!)^2}{(1+2k)!}\right]|a_{k}|^2.
\end{equation*}
Appealing to standard asymptotic expansions of factorials, we find that
\begin{equation}
\frac{(k!)^2}{(1+2k)!}=\left(\begin{array}{c} 2k\\k \end{array}\right)^{-1}\frac{1}{1+2k}\asymp 4^{-k}k^{-1/2},
\label{factest}
\end{equation}
so that 
$\|f\|^2_{\alpha}\asymp \sum_{k=0}^{\infty}4^{-k}(k+1)^{\alpha-1/2}|a_k|^2$.
The latter expression coincides with (a multiple of) the norm in $d_{\alpha-1/2}$ of 
$g(z)=\sum_{k=0}^{\infty}2^{-k/2}a_kz^k$, which can be identified with the 
restriction $\mathcal{R}(f)$. 
\end{proof}
\begin{ex}
 Consider the polynomial $f(z_1,z_2)=1-2z_1z_2$.
We check that
\[\mathcal{Z}(f)\cap \overline{\mathbb{B}^2}=\mathcal{Z}(f)\cap \mathbb{S}^2=
\left\{\frac{1}{\sqrt{2}}\left(e^{i\theta},e^{-i\theta}\right)\colon \theta\in [0,2\pi)\right\},\]
which can be viewed as a curve contained in the unit sphere. 
By Theorem \ref{diagspaces}, $f$ is cyclic in $\mathcal{D}_{\alpha}$ for $\alpha\leq 3/2$, and not cyclic for $\alpha>3/2$. This is analogous to the situation in the bidisk, where $1-z_1z_2$ fails to be cyclic in $\mathfrak{D}_{\alpha}$, $\alpha>1/2$. 

Armed with this result, we deduce 
cyclicity, for the same range of $\alpha$, also for 
$f^{\flat}(z_1,z_2)=1-z_1^2+z_2^2$ 
since $f^{\flat}=f\circ U$ for the unitary \eqref{unitary}. In this case $f^{\flat}\notin \mathcal{J}_{\alpha}$, but the zero set is again a curve
$\mathcal{Z}(f^{\flat}\cap\mathbb{S}^2)=\{(\sin \theta, i\cos\theta)\colon \theta \in [0,2\pi))\}$, and the cyclicity properties of $f^{\flat}$ reflect this.
\end{ex}
\begin{rem}
The situation in the ball thus mirrors the bidisk case: 
when the spaces are algebras ($\alpha>1$ in the bidisk, $\alpha>2$ in 
the ball) functions have to be non-vanishing 
in the closure of the domain in order to be cyclic, but there is 
a regime ($\alpha>1/2$ in the bidisk, and $\alpha>3/2$ 
in the ball) where some polynomials that vanish on the Shilov boundary are 
cyclic, while others are not. 
In view of our examples, it is natural to ask for a characterization 
of cyclic polynomials analogous to that achieved in \cite{BKKLSS}. This 
looks like a harder problem 
in the ball because of the absence of determinantal representations.\end{rem}

\section{Cauchy transforms and capacity of zero sets}
One way of identifying non-cyclic vectors in Dirichlet-type spaces
is to employ the following scheme due to Brown and Shields. 
Put a measure on $\mathcal{Z}(f^*)$, the zero set in the sphere of the radial limits of a given $f\in \mathcal{D}_{\alpha}$, and consider its Cauchy 
transform. By examining the integral version of the dual pairing 
discussed earlier, and arguing as in \cite{BS84}, 
we find that the Cauchy transform under consideration 
annihilates all polynomial multiples of $f$. Hence, by the Hahn-Banach theorem, we can conclude
that $[f]\neq \mathcal{D}_{\alpha}$---provided we know that the Cauchy transform 
belongs to $\mathcal{D}_{-\alpha}$ (regarded here as the dual  
of $\mathcal{D}_{\alpha}$). 
\begin{ex}
Here is a trivial example. Consider the function $f(z_1,z_2)=1-z_1$, which is cyclic for $\alpha\leq 2$. The natural Cauchy transform 
associated with its zero set is $C[\delta_{(1,0)}](z_1,z_2)=(1-z_1)^{-2}=\sum_{k=0}^{\infty}(k+1)z_1^k$,
which is in $\mathcal{D}_{-\alpha}$ for every $\alpha>2$, and since $\langle f, C[\delta_{(1,0)}]\rangle=0$ by direct computation,
this reproves that $f$ is not cyclic when $\alpha>2$.

A more interesting example is the Cauchy transform of the integration current associated with $\mathcal{Z}(1-2z_1z_2)$: in that case
\[C[\mu_{\mathcal{Z}}](z_
1,z_2)=\frac{1}{2\pi}\int_{0}^{2\pi}\frac{d\theta}{(1-(z_1e^{i\theta}-z_2e^{-i\theta})/\sqrt{2})^2}.\]
We expand the integrand and obtain
\[C[\mu_{\mathcal{Z}}](z_1,z_2)=\sum_{k=0}^{\infty}
2^{-k/2}(k+1)\sum_{j=0}^k\left(\begin{array}{c}k\\j\end{array}\right)\left(\int_{\mathbb{T}}e^{i(2j-k)\theta}\frac{d\theta}{2\pi}\right)z_1^jz_2^{k-j}.\]
The integral on the right-hand side is equal to zero unless $2j=k$, in which 
case its value is equal to $1$. Hence
\[C[\mu_{\mathcal{Z}}](z_1,z_2)=\sum_{k=0}^{\infty}\frac{(1+2k)!}{(k!)^2}\left(\frac{z_1z_2}{2}\right)^k=\frac{1}{(1-2z_1z_2)^{3/2}}.\]
Using the asymptotic expansion of the ratio of factorials in \eqref{factest}, we find that 
\[\|C[\mu_{\mathcal{Z}}]\|^2_{-\alpha}\asymp \sum_{k=0}^{\infty}(k+1)^{-\alpha+1/2},\]
We have convergence, and hence non-cyclicity of $f$, provided $\alpha>3/2$. 
%(Note that there is a certain degree of smoothing in the Cauchy transform when its generating measure is smeared out.)
\end{ex}
For a Borel measure $\mu$ on $\mathbb{S}^2$ 
and $m,n\in \mathbb{N}$,
\[\mu^*(m,n)=\int_{\mathbb{S}^2}\zeta_1^m\zeta_2^nd\mu(\zeta_1,\zeta_2)
\quad \textrm{and}\quad 
\bar{\mu}^*(m,n)=\int_{\mathbb{S}^2}\overline{\zeta_1}^m\overline{\zeta_2}^n
d\mu(\zeta_1,\zeta_2).\]
If $\mu$ is a real measure, then $\bar{\mu}^*(m,n)=\overline{\mu^*(m,n)}$.
\begin{lem}\label{cauchylemma}
If $\mu$ is a Borel measure on $\mathbb{S}^2$, then 
\[\|C[\mu]\|^2_{-\alpha}\asymp \sum_{k=0}^{\infty}\sum_{j=0}^k(k+1)^{1-\alpha}
\left(\begin{array}{c}k\\j\end{array}\right)|\mu^*(j,k-j)|^2.\]
\end{lem}
\begin{proof}
The Cauchy kernel has expansion 
$(1-\langle z, \bar{\zeta}\rangle)^{-2}=\sum_{k=0}^{\infty}(k+1)
\langle z, \zeta \rangle^k$, and expanding the powers of the 
inner product, 
we obtain that
\[C[\mu](z_1,z_2)=\sum_{k=0}^{\infty}
\sum_{j=0}^k(k+1)\left(\begin{array}{c}k\\j\end{array}\right)\bar{\mu}^*(j,k-j)
z_1^jz_2^{k-j}.\]
Computing the $\mathcal{D}_{-\alpha}$ norm, we find that 
\begin{multline*}
\|C[\mu]\|_{-\alpha}\\=
\sum_{k=0}^{\infty}\sum_{j=0}^k\frac{j!(k-j)!}{(1+j+k-j)!}
(2+j+k-j)^{-\alpha}
\left(\begin{array}{c}k\\j\end{array}\right)^2(k+1)^2|\mu^*(j,k-j)|^2\\
\asymp \sum_{k=0}^{\infty}\sum_{j=0}^k(k+1)^{2-\alpha}
\left(\begin{array}{c}k\\j\end{array}\right) \frac{k!}{j!(k-j)!}
\frac{j!(k-j)!}{(1+k)!}|\mu^*(j,k-j)|^2\\
\asymp \sum_{k=0}^{\infty}\sum_{j=0}^k(k+1)^{1-\alpha}\left(\begin{array}{c}k\\j\end{array}\right)|\mu^*(j,k-j)|^2.
\end{multline*}
\end{proof}

In the disk and bidisk, capacitary conditions 
provide us with a way of checking whether 
Cauchy transforms of measures supported on the zero set of a function are
in the right dual spaces. In the case of the 
bidisk, the appropriate capacities live on the distinguished boundary, and 
are of product type. There is an analogous notion of capacity for the ball, 
adapted to (the square of) the natural 
distance function of the sphere,
\[d(\zeta,\eta)=|1-\langle \zeta, \eta\rangle|^{1/2}, \quad 
\zeta=(\zeta_1,\zeta_2),\, \eta=(\eta_1,\eta_2)\in \mathbb{S}^2.\]
Applying a positive kernel to a distance function is of course a standard construction in potential theory (see eg. \cite{EKMRBook}).
\begin{defn}[Riesz $\alpha$-capacity for the sphere.]
Let $\mu$ be a Borel probability measure supported on some 
Borel set $E\subset \mathbb{S}^2$. Set 
\[h_{\alpha}(t)=\left\{\begin{array}{cc}
t^{\alpha-2},& \alpha \in (0,2),\\ 
\log(e/t), &\alpha=2 \end{array}\right. .
\]
The {\it 
Riesz $\alpha$-energy} of $\mu$ is given by
\[I_{\alpha}[\mu]=\int_{\mathbb{S}^2}\int_{\mathbb{S}^2}
h_{\alpha}(|1-\langle \zeta, \eta\rangle|)d\mu(\zeta)d\mu(\eta).\]
The {\it Riesz $\alpha$-capacity} of a Borel set $E$ is 
$\mathrm{cap}_{\alpha}(E)=\inf \{I_{\alpha}[\mu]\colon \mu \in 
\mathcal{P}(E)\}^{-1}$.
The set $E$ is said to have $\alpha$-capacity $0$ if there does not exist any probability measure, 
with support in $E$, that has finite $\alpha$-energy.
\end{defn}
These anisotropic capacities,
as well as closely related variants, appear in works of Ahern and Cohn, and 
others, on 
function theory on the ball, see eg. \cite{AC89, PR96}; the 
parametrizations are sometimes slightly different. In particular, 
one can show (see \cite{AC89, CV95}) that any $f\in \mathcal{D}_{\alpha}$ 
has finite radial limits (and even $K$-limits) $f^*$ on $\mathbb{S}^2$, except
possibly on a set $E_f$ having $\mathrm{cap}_{\alpha}(E_f)=0$. 

\begin{lem}\label{energylemma}
Suppose $\mu$ is a Borel probability measure on $\mathbb{S}^2$. Then
\begin{equation}
I_2[\mu]=1+\sum_{k=1}^{\infty}\frac{1}{k}\sum_{j=0}^k
\left(\begin{array}{c}k\\j \end{array}\right)|\mu^*(j,k-j)|^2
\label{energyformula}
\end{equation}
\end{lem} 
\begin{proof}
Noting that $\langle \zeta, \eta\rangle \in \overline{\mathbb{D}}$ when $\zeta, \eta \in \mathbb{S}^2$, we expand $h_{2}(|1-\langle \zeta, \eta\rangle|)$ as 
\[\log \frac{e}{|1-\langle \zeta, \eta\rangle|}=
1+\mathrm{Re}\left(\log \frac{1}{1-\langle \zeta, \eta\rangle}\right)=
1+\mathrm{Re}\sum_{k=1}^{\infty}\frac{1}{k}\langle \zeta, \eta\rangle^k.\]
Thus, for $0<r<1$, we can use the positivity of $\mu$ to write
\[\int_{\mathbb{S}^2}\int_{\mathbb{S}^2}\log\frac{e}{|1-r\langle\zeta, \eta\rangle|}
d\mu(\zeta)d\mu(\eta)=1+
\mathrm{Re}\sum_{k=1}^{\infty}\frac{r^k}{k}\int_{\mathbb{S}^2}\int_{\mathbb{S}^2}\langle
\zeta, \eta\rangle^kd\mu(\zeta)d\mu(\eta),\]
and the powers of $\langle \zeta,\eta\rangle$ can further be expanded using the binomial 
formula. 

Using the fact that we have $\mu^*(m,n)\bar{\mu}^*(m,n)=|\mu^*(m,n)|^2$,
the proof now proceeds along the lines of \cite[Thm. 2.4.4]{EKMRBook}, 
with Fourier coefficients of a measure replaced by $\mu^*(m,n)$.
\end{proof}

Combining Lemmas \ref{cauchylemma} and \ref{energylemma}, we arrive at 
a necessary condition for cyclicity in the Dirichlet space
(see \cite{BS84, BCLSSII13} for the disk and bidisk versions).
\begin{thm}\label{BSthm}
Suppose $f\in \mathcal{D}$ has 
$\mathrm{cap}_{2}(\mathcal{Z}(f^*))>0$. 
Then $f$ is {\it not} cyclic in the Dirichlet space $\mathcal{D}$.
\end{thm}
\begin{rem}
For $\alpha$-capacities, the counterpart to the energy formula 
\eqref{energyformula} is no longer a strict equality; rather 
(see \cite{PR96} for computations of Fourier coefficients associated with 
the kernels $h_{\alpha}$), 
\[I_{\alpha}[\mu]\asymp \sum_{k=0}^{\infty}\sum_{j=0}^k
(k+1)^{1-\alpha}\left(\begin{array}{c}k\\j \end{array}\right)
|\mu^*(j,k-j)|^2.\]
A comparison with Lemma \ref{cauchylemma} shows that 
$\mathrm{cap}_{\alpha_0}(\mathcal{Z}(f^*))>0$ implies that 
$f$ is not cyclic in $\mathcal{D}_{\alpha}$ for $\alpha\geq 
\alpha_0$.
\end{rem}
\subsection*{Acknowledgment}
The author thanks Brett Wick for pointing out several useful references, and Daniel Seco for helpful remarks on an 
earlier version of this note.

% ------------------------------------------------------------------------
\end{document}